\documentclass[a4paper,11pt]{amsart}

\usepackage[left=2.7cm,right=2.7cm,top=3.5cm,bottom=3cm]{geometry}
\usepackage{amsthm,amssymb,amsmath,amsfonts,mathrsfs,amscd}
\usepackage[latin1]{inputenc}
\usepackage[all]{xy}

\theoremstyle{definition}
\newtheorem{defi}{Definition}[section]
\newtheorem{assumption}[defi]{Assumption}

\theoremstyle{remark}
\newtheorem{rem}[defi]{Remark}

\theoremstyle{plain}
\newtheorem{lemma}[defi]{Lemma}
\newtheorem{coro}[defi]{Corollary}
\newtheorem{teo}[defi]{Theorem}
\newtheorem{prop}[defi]{Proposition}

\newcommand{\Aut}{\mathrm{Aut}}

\newcommand{\Hom}{\mathrm{Hom}}

\DeclareMathOperator{\GL}{GL}

\newcommand{\M}{\mathrm{M}}

\newcommand{\ad}{\mathrm{ad}}

\newcommand{\Q}{\mathbb Q}

\newcommand{\Z}{\mathbb Z}

\newcommand{\F}{\mathbb F}

\newcommand{\fr}{\mathfrak}
\newcommand{\cl }{\mathcal}

\begin{document}

\title[An irreducibility criterion for group representations]{An irreducibility criterion for group representations, with arithmetic applications}

\author{Matteo Longo and Stefano Vigni}

\address{M. L.: Dipartimento di Matematica Pura e Applicata, Università di Padova, Via Trieste 63, 35121 Padova, Italy}
\email{mlongo@math.unipd.it}
\address{S. V.: Departament de Matemàtica Aplicada II, Universitat Politècnica de Catalunya, C. Jordi Girona 1-3, 08034 Barcelona, Spain}
\email{stefano.vigni@upc.edu}


\begin{abstract}
We prove a criterion for the irreducibility of an integral group representation $\rho$ over the fraction field of a noetherian domain $R$ in terms of suitably defined reductions of $\rho$ at prime ideals of $R$. As applications, we give irreducibility results for universal deformations of residual representations, with a special attention to universal deformations of residual Galois representations associated with modular forms of weight at least $2$.
\end{abstract}

\maketitle

\section{Introduction}

Let $G$ be a group, let $R$ be a noetherian integral domain with fraction field $K$ and consider a representation
\[ \rho_K:G\longrightarrow\GL_d(K) \]
of $G$ over $K$. A classical problem in representation theory is to find suitable conditions under which $\rho_K$ is integral over $R$, i.e., under which there exists a representation
\[ \rho_R:G\longrightarrow\GL_d(R) \]
of $G$ over $R$ which is equivalent (i.e., conjugated) to $\rho_K$. If this is the case, one says that $\rho_K$ can be \emph{realized over} $R$ via $\rho_R$ (see \S \ref{terminology-subsection} for more precise definitions). At least when the group $G$ is finite and the ring $R$ is a Dedekind domain (which we do not require), dealing with this and related questions amounts to studying $R$-orders and lattices in non-commutative $K$-algebras (see, e.g., \cite[Ch. 3]{cr}). Here we propose to tackle a different problem: we assume that $\rho_K$ can be realized over $R$ via a representation $\rho_R$ as above and we look for properties of $\rho_R$ which guarantee that $\rho_K$ is irreducible. 

More precisely, in \S \ref{reductions-subsec} we define reductions $\bar\rho_\fr p$ of $\rho_R$ at prime ideals $\fr p$ of $R$; these are representations of $G$ over the residue fields of the localizations of $R$ at the primes $\fr p$, and our goal is to relate the irreducibility of the $\bar\rho_\fr p$ to that of $\rho_K$. As a motivation, consider the toy case in which $R$ is a discrete valuation ring with maximal ideal $\fr m$: in this situation it is easy to see (essentially by applying Nakayama's lemma) that $\rho_K$ is irreducible if $\bar\rho_\fr m$ is. 

Of course, in the general setting where $R$ is allowed to be an arbitrary noetherian domain one expects extra complications to arise. Nevertheless, the main result of this note (Theorem \ref{thm1}) shows that a criterion of this sort is still valid; in fact, we can prove 

\begin{teo}
If $\bar\rho_\fr p$ is irreducible for a set of prime ideals $\fr p$ of $R$ with trivial intersection then $\rho_K$ is irreducible. 
\end{teo}

It is worthwhile to remark that the group $G$ is \emph{arbitrary} (in particular, it need not be finite). This result has a number of consequences (cf. \S \ref{main-subsection} and \S \ref{regular-subsec}). As an example, in Theorem \ref{thm2} we extend the implication recalled above for discrete valuation rings to a much larger class of local domains; namely, we prove that if $R$ is a regular local ring with maximal ideal $\fr m$ then $\rho_K$ is irreducible if $\bar\rho_\fr m$ is.  

Section \ref{arithmetic-section} ends the paper with applications of our algebraic results to arithmetic contexts. In particular, in \S \ref{universal-subsec} we deal with deformations of residual representations in the sense of Mazur (\cite{mazur}) and prove that an irreducible residual representation admits irreducible universal deformations when the deformation problem for it is unobstructed (see Proposition \ref{irreducible-def-prop} for an accurate statement). Finally, in \S \ref{modular-subsec} we specialize this irreducibility result to the case of residual modular Galois representations. In this setting, using results of Weston (\cite{weston}), we show that if $f$ is a newform and $K_f$ is the number field generated by its Fourier coefficients then for infinitely many (in a strong sense) primes $\lambda$ of $K_f$ the residual Galois representation $\bar\rho_{f,\lambda}$ attached by Deligne to $f$ and $\lambda$ has irreducible universal deformation (see Proposition \ref{modular-prop} for a precise formulation).

\vskip 2mm

\noindent\emph{Acknowledgements.} We would like to thank Luis Dieulefait and Marco Seveso for helpful discussions and comments. The second author acknowledges the warm hospitality of the Centre de Recerca Matemàtica (Bellaterra, Spain) during Autumn 2009, when this work was completed. 

\section{Representations over lattices and reductions}

In this section we review the basic definitions concerning integral group representations and their reductions modulo prime ideals.

\subsection{Terminology and auxiliary results} \label{terminology-subsection}

As in the introduction, let $R$ be a noetherian domain with fraction field $K$ and let $G$ be a group. Let $V$ be a vector space over $K$ of dimension $d\geq1$ and let
\begin{equation} \label{def-rho}
\rho_V:G\longrightarrow\Aut_K(V)\simeq\GL_d(K)
\end{equation}
be a representation of $G$ in the $K$-vector space of $K$-linear automorphisms of $V$. If $L$ is an $R$-submodule of $V$ write $KL$ for the $K$-subspace of $V$ generated by $L$.

\begin{defi}
An \emph{$R$-lattice} (or simply a \emph{lattice}) of $V$ is a finitely generated $R$-submodule $L$ of $V$ such that $KL=V$.
\end{defi}

In other words, the lattice $L$ is a finitely generated $R$-submodule of $V$ which contains a basis of $V$ over $K$.

\begin{rem} \label{lattice-rem}
1) It can be shown (see \cite[Ch. VII, \S 4.1]{b}) that if $L$ is an $R$-lattice of $V$ then there exists a free $R$-submodule of $L$ of rank $d$.

2) If the group $G$ is finite then $G$-stable lattices in $V$ always exist. However, such lattices are not necessarily free over $R$ (see, e.g., \cite[p. 550]{serre}). 
\end{rem}

If the lattice $L$ is stable for the action of $G$ (i.e., is a left $R[G]$-module) then we can define a representation
\[ \rho_L:G\longrightarrow\Aut_R(L) \]
of $G$ in the $R$-module of $R$-linear automorphisms of $L$. More generally, if $A$ is any $R$-algebra (with trivial $G$-action) we can consider the representation
\[ \rho_{L,A}:G\longrightarrow\Aut_A(L\otimes_RA) \]
of $G$ in the $A$-module of $A$-linear automorphisms of $L\otimes_RA$ which is obtained by extending $\rho_L$ by $A$-linearity.

Now we can give the following important

\begin{defi}
The representation $\rho_V$ is \emph{integral over $R$} if there exists a $G$-stable lattice $L$ of $V$. If this is the case then the isomorphism $L\otimes_RK\simeq V$ is $G$-equivariant and we say that $V$ can be \emph{realized over} $L$.
\end{defi}

Notice that, in light of part 1) of Remark \ref{lattice-rem}, assuming the lattice $L$ to be free over $R$ is not too serious a restriction; in fact, in order not to burden our exposition with unenlightening technicalities, the main results of this paper will be proved under this condition.\\

To avoid ambiguities, we recall some standard terminology. The representation $\rho_V$ is said to be \emph{irreducible} if the only $K$-subspaces of $V$ which are invariant for $G$ under $\rho_V$ are $\{0\}$ and $V$; if $\rho_V$ is understood, we also say that the left $K[G]$-module $V$ is \emph{irreducible}. More generally, a left $R[G]$-module $L$ which is finitely generated over $R$ is said to be \emph{irreducible} if $L$ does not contain any (necessarily finitely generated, since $R$ is noetherian) $R$-submodule $M$ which is $G$-stable and such that $KM$ is neither trivial nor equal to $KL$; in this case, we also say that the representation $\rho_L$ defined as above is \emph{irreducible}.

\begin{prop} \label{prop1}
With notation as before, let $\rho_V$ be realized over $L$. Then $\rho_V$ is irreducible if and only if $\rho_L$ is.
\end{prop}

\begin{proof} Assume first that $\rho_V$ is irreducible and let $M$ be a finitely generated $R$-submodule of $V$ which is $G$-stable. The $K$-subspace $KM$ is then a $K[G]$-submodule of $V$, so it must be either trivial or equal to $V$, hence $\rho_L$ is irreducible. Conversely, suppose that $\rho_L$ is irreducible. If $\rho_V$ were not irreducible then we could find a $K[G]$-submodule $W$ of $V$ such that $W\not=\{0\}$ and $W\not=V$. Set $M:=L\cap W$. Then $M$ is a finitely generated $R$-submodule of $L$ (because $R$ is noetherian) which is $G$-stable and such that
\[ KM=(KL)\cap W=V\cap W=W, \]
which contradicts the irreducibility of $\rho_L$. \end{proof}

Proposition \ref{prop1} makes it possible to study the irreducibility of $\rho_V$ in terms of that of $\rho_L$, and this will be the underlying theme of the rest of the paper.

\subsection{Reductions modulo prime ideals} \label{reductions-subsec}

From now on let $\rho_V$ be a representation of the group $G$ which is realized over the lattice $L$. For any prime ideal $\fr p$ of $R$ write $R_\fr p$ for the localization of $R$ at $\fr p$ and let
\[ k_\fr p:=R_\fr p/\fr pR_\fr p,\qquad\pi_\fr p:R_\fr p\longrightarrow k_\fr p \]
be the residue field of $R_\fr p$ and the canonical quotient map, respectively. Define $L_\fr p:=R_{\fr p}L$ (as submodule of $V$) and set $\rho_{L,\fr p}:=\rho_{L,R_\fr p}$, for short. Observe that $L_\fr p$ is the localization of $L$ at $\fr p$, as suggested by the notation, so that there is a canonical isomorphism $L_{\fr p}\simeq L\otimes_RR_{\fr p}$. Moreover, $L_\fr p$ is an $R_\fr p$-lattice of $V$. We also have a residual representation $\bar\rho_{L,\fr p}$ on the field $k_\fr p$ which is defined as the composition
\[ \bar\rho_{L,\fr p}:G\xrightarrow{\rho_{L,\fr p}}\Aut_{R_\fr p}(L_\fr p)\overset{\pi_\fr p}{\longrightarrow}\Aut_{k_\fr p}(L_\fr
p/\fr pL_\fr p). \]
In particular, taking as $\fr p$ the trivial ideal $(0)$ of $R$ gives $\rho_{L,(0)}=\bar\rho_{L,(0)}=\rho_V$. The notion of irreducibility for the representations $\bar\rho_{L,\fr p}$ is the obvious one.

To motivate the main theorem of this note, in the next subsection we recall a classical result over discrete valuation rings.

\subsection{Discrete valuation rings}

Suppose now that $\cl O$ is a discrete valuation ring and fix a generator $\pi$ of its maximal ideal $\wp$. Let $F$ denote the fraction field of $\cl O$. In this case, since $\cl O$ is a principal ideal domain, every $\cl O$-lattice $L$ in a finite-dimensional $F$-vector space $V$ is free over $\cl O$. Of course, $\cl O_\wp=\cl O$, hence $L_\wp=L$ and the reduced representation $\bar\rho_{L,\wp}$ is equal to the composition
\[ \bar\rho_{L,\wp}:G\overset{\rho_L}{\longrightarrow}\Aut_\cl O(L)\overset{\pi_\wp}{\longrightarrow}\Aut_{k_\wp}(L/\wp L) \]
where $k_\wp:=\cl O/\wp$ is the residue field of $\cl O$. The result we are about to state is well known, but we recall it here because it represents a motivation for the theorem that will be proved in the subsequent section.

\begin{prop} \label{dvr-prop}
Let $\cl O$, $\rho_V$ and $\rho_L$ be as above. If $\bar\rho_{L,\wp}$ is irreducible then $\rho_V$ is irreducible.
\end{prop}

\begin{proof} By Proposition \ref{prop1}, we can equivalently prove that $\rho_L$ is irreducible. Arguing by contradiction, suppose that $M$ is a $G$-stable $\cl O$-submodule of $L$ such that $KM$ is neither trivial nor equal to $V$; in particular, $M\not=L$. It is easy to see that there exists $n\in\mathbb N$ such that
\[ M\subset\wp^nL,\qquad M\not\subset\wp^{n+1}L. \]
Hence, at the cost of multiplying $M$ by a suitable power of $\pi$, we can assume that $M\not\subset\wp L$. Since $M\not=L$, Nakayama's lemma ensures that $M+\wp L\not=L$ (see, e.g., \cite[Corollary 2.7]{am}), hence the image of $M$ in $L/\wp L$ is non-trivial and strictly smaller than $L/\wp L$. This contradicts the irreducibility of $\bar\rho_{L,\wp}$. \end{proof}

The question of finding an analogue of Proposition \ref{dvr-prop} when the discrete valuation ring $\cl O$ is replaced by an arbitrary noetherian domain has been the starting point of our investigation. Theorem \ref{thm1} gives a reasonable answer to this problem for a large class of rings.

\section{The irreducibility theorem}

This section is devoted to the proof of the main result of this paper, which is given in \S \ref{main-subsection}.

\subsection{Algebraic preliminaries}

For lack of an explicit reference, we give a proof of the following elementary result in linear algebra.

\begin{lemma} \label{tensor-lemma}
Let $A$ be a ring, let $M$ be a free $A$-module of finite rank and let ${\{N_\lambda\}}_{\lambda\in\Lambda}$ be a family of $A$-modules. There is a canonical isomorphism
\[ \bigg(\prod_{\lambda\in\Lambda}N_\lambda\bigg)\otimes_AM\simeq\prod_{\lambda\in\Lambda}(N_\lambda\otimes_AM) \]
of $A$-modules.
\end{lemma}
\begin{proof} Let $\{e_1,\dots,e_t\}$ be a basis of $M$ over $A$, so that we can write
\[ M=\bigoplus_{i=1}^tAe_i. \]
There are canonical isomorphisms of $A$-modules
\begin{equation} \label{eq1}
\bigg(\prod_{\lambda\in\Lambda}N_\lambda\bigg)\otimes_AM\simeq\bigoplus_{i=1}^t\bigg[\bigg(\prod_{\lambda\in\Lambda}N_\lambda\bigg)\otimes_AAe_i\bigg]\simeq\bigoplus_{i=1}^t\bigg(\prod_{\lambda\in\Lambda}N_\lambda\bigg)e_i
\end{equation}
and
\begin{equation} \label{eq2}
\prod_{\lambda\in\Lambda}(N_\lambda\otimes_AM)\simeq\prod_{\lambda\in\Lambda}\bigg[\bigoplus_{i=1}^t\bigl(N_\lambda\otimes_AAe_i\bigr)\bigg]\simeq\prod_{\lambda\in\Lambda}\bigg(\bigoplus_{i=1}^tN_\lambda e_i\bigg).
\end{equation}
But there is also a canonical isomorphism of $A$-modules
\[ \begin{array}{ccc}
   \displaystyle{\bigoplus_{i=1}^t\bigg(\prod_{\lambda\in\Lambda}N_\lambda\bigg)e_i} & \overset{\simeq}{\longrightarrow} & \displaystyle{\prod_{\lambda\in\Lambda}\bigg(\bigoplus_{i=1}^tN_\lambda e_i\bigg)}\\[6mm]
   \bigl({(n_{1,\lambda})}_\lambda e_1,\dots,{(n_{t,\lambda})}_\lambda e_t\bigr) & \longmapsto & \bigl((n_{1,\lambda}e_1,\dots,n_{t,\lambda}e_t)\bigr)_{\lambda\in\Lambda},
   \end{array} \]
and the claim follows by combining \eqref{eq1} and \eqref{eq2}. \end{proof}

We use this result to prove

\begin{prop} \label{intersection-tensor-prop}
Let $A$ be a ring, let $M$ be a free $A$-module of finite rank, let $N$ be an $A$-module and let ${\{N_\lambda\}}_{\lambda\in\Lambda}$ be a family of $A$-submodules of $N$. There is a canonical isomorphism
\begin{equation} \label{int-tensor-eq}
\bigg(\bigcap_{\lambda\in\Lambda}N_\lambda\bigg)\otimes_AM\simeq\bigcap_{\lambda\in\Lambda}(N_\lambda\otimes_AM)
\end{equation}
of $A$-submodules of $N\otimes_AM$.
\end{prop}

\begin{proof} First of all, since $M$ is flat over $A$, for all $\lambda\in\Lambda$ there is a canonical isomorphism
\begin{equation} \label{quotient-tensor-eq}
(N/N_\lambda)\otimes_AM\simeq(N\otimes_AM)/(N_\lambda\otimes_AM)
\end{equation}
of $A$-modules. Define the $A$-linear map
\[ \begin{array}{ccc}
   N & \overset{\varphi}{\longrightarrow} & \displaystyle{\prod_{\lambda\in\Lambda}(N/N_\lambda)}\\[6mm]
   x & \longmapsto & \bigl({[x]}_\lambda\bigr)_{\lambda\in\Lambda}
   \end{array} \]
where ${[x]}_\lambda$ denotes the natural image of $x\in N$ in $N/N_\lambda$. Then the sequence
\[ 0\longrightarrow\bigcap_{\lambda\in\Lambda}N_\lambda\longrightarrow N\overset{\varphi}{\longrightarrow}\prod_{\lambda\in\Lambda}(N/N_\lambda) \]
is exact and, since $M$ is free of finite rank over $A$, by Lemma \ref{tensor-lemma} and isomorphism \eqref{quotient-tensor-eq} the same is true of the sequence
\[ 0\longrightarrow\bigg(\bigcap_{\lambda\in\Lambda}N_\lambda\bigg)\otimes_AM\longrightarrow N\otimes_AM\longrightarrow\prod_{\lambda\in\Lambda}(N\otimes_AM)/(N_\lambda\otimes_AM), \]
which proves the proposition. \end{proof}

\subsection{Proof of the irreducibility criterion} \label{main-subsection}

The main result on the irreducibility of $\rho_V$ we wish to prove is the following

\begin{teo} \label{thm1}
Let $R$ be a noetherian domain with field of fractions $K$, let $\rho_V$, $\rho_L$ be as above and assume that the lattice $L$ is free over $R$. Suppose that there exists a set $\cl S$ of prime ideals of $R$ such that
\begin{itemize}
\item[(i)] $\bigcap_{\fr p\in\cl S}\fr p=\{0\}$;
\item[(ii)] $\bar\rho_{L,\fr p}$ is irreducible for every $\fr p\in\cl S$.
\end{itemize}
Then $\rho_V$ is irreducible.
\end{teo}

Of course, if the trivial ideal of $R$ belongs to $\cl S$ there is nothing to prove, so we can implicitly assume that $(0)\not\in\cl S$.

\begin{proof} Let $W\neq V$ be a non-trivial $G$-stable $K$-subspace of $V$ of dimension $m\geq1$ and define
\[ M:=L\cap W. \]
It is clear that $M$ is an $R$-submodule of $L$ which is $G$-stable. Moreover, $M$ is a lattice of $W$ by \cite[Ch. VII, \S 4, Proposition 3 (ii)]{b}.

Let now $\fr p$ be a prime ideal of $R$ such that $\bar\rho_{L,\fr p}$ is irreducible and set
\[ M_\fr p:=R_\fr pM. \]
Then $M_\fr p$ is an $R_\fr p$-submodule of $L_\fr p$ which is $G$-stable. Furthermore, $M_\fr p\not=L_\fr p$ because otherwise we would have
\[ W=KM_\fr p=KL_\fr p=V, \]
contrary to our assumption. By Nakayama's lemma, this implies that $M_\fr p+\fr pL_\fr p\not=L_\fr p$, hence the image of $M_\fr p$ in $L_\fr p/\fr pL_\fr p$ is strictly smaller than $L_\fr p/\fr pL_\fr p$. But $\bar\rho_{L,\fr p}$ is irreducible, so there is an inclusion
\begin{equation} \label{loc-inclusion-eq}
M_\fr p\subset\fr pL_\fr p.
\end{equation}
Since $M$ is a lattice in $W$, it contains a free $R$-module $M'$ of rank $m$ (cf. part 1) of Remark \ref{lattice-rem}). To prove that $\rho_V$ is irreducible we want to show that $M'=\{0\}$, a fact that would contradict the non-triviality of $W$. Since $M'$ injects naturally into $M_\fr p$, inclusion \eqref{loc-inclusion-eq} yields an inclusion
\begin{equation} \label{inclusion-eq}
M'\subset\fr pL_\fr p\cap L=\fr pL
\end{equation}
for all prime ideals $\fr p$ of $R$ such that $\bar\rho_{L,\fr p}$ is irreducible. Observe that the equality appearing in \eqref{inclusion-eq} is an easy consequence of the fact that $L$ is free (of finite rank) over $R$. In light of condition (ii), to show that $M'=\{0\}$ it is then enough to show that
\begin{equation} \label{intersection-eq}
\bigcap_{\fr p\in\cl S}\fr pL=\{0\}.
\end{equation}
Since $L$ is free (hence flat) over $R$, for any ideal $\fr a$ of $R$ there is a canonical isomorphism
\begin{equation} \label{ideal-isom-eq}
\fr aL\simeq\fr a\otimes_RL
\end{equation}
of $R$-modules. Moreover, since $L$ has finite rank over $R$, if $\cl A$ is a family of ideals of $R$ then Proposition \ref{intersection-tensor-prop} provides a canonical isomorphism
\begin{equation} \label{tensor-eq}
\bigg(\bigcap_{\fr a\in\cl A}\fr a\bigg)\otimes_RL\simeq\bigcap_{\fr a\in\cl A}(\fr a\otimes_RL)
\end{equation}
of $R$-modules. Applying \eqref{ideal-isom-eq} and \eqref{tensor-eq} to $\cl A=\cl S$ yields an isomorphism
\[ \bigcap_{\fr p\in\cl S}\fr pL\simeq\bigg(\bigcap_{\fr p\in\cl S}\fr p\bigg)\otimes_RL \]
of $R$-modules, and \eqref{intersection-eq} follows immediately from condition (i). \end{proof}

\begin{rem}
1) Since $R$ is a domain, any set of non-trivial prime ideals of $R$ satisfying condition (i) of Theorem \ref{thm1} is necessarily infinite. Thus the irreducibility criterion of Theorem \ref{thm1} is interesting only for rings with infinitely many prime ideals.

2) For an example of a different representation-theoretic context in which irreducibility modulo prime ideals plays an important role see \cite[\S 3]{gross}.  
\end{rem}

It is useful to slightly reformulate condition (ii) of the above theorem. For every prime ideal $\fr p$ of $R$ define the representation $\tilde\rho_{L,\fr p}$ over the noetherian domain $R/\fr p$ as the composition
\[ \tilde\rho_{L,\fr p}:G\overset{\rho_L}\longrightarrow\Aut_R(L)\longrightarrow\Aut_{R/\fr p}(L/\fr pL), \]
where the second map is induced by the canonical projection $R\rightarrow R/\fr p$. If $\text{frac}(R/\fr p)$ is the fraction field of $R/\fr p$ then there is a natural identification
\[ \text{frac}(R/\fr p)=k_\fr p \]
and $L/\fr pL$ is an $R/\fr p$-lattice in $L_{\fr p}/\fr pL_{\fr p}$, hence Proposition \ref{prop1} ensures that condition (ii) in Theorem \ref{thm1} is equivalent to
\begin{itemize}
\item[(ii')] $\tilde\rho_{L,\fr p}$ is irreducible for every $\fr p\in\cl S$.
\end{itemize}
For the rest of the paper we make the following

\begin{assumption}
The lattice $L$ is \emph{free} over $R$.
\end{assumption}

The next result is an easy consequence of the previous theorem.

\begin{coro} \label{coro1}
With notation as above, if $\bar\rho_{L,\fr p}$ is irreducible for infinitely many height $1$ prime ideals $\fr p$ of $R$ then $\rho_V$ is irreducible.
\end{coro}

\begin{proof} By Theorem \ref{thm1}, it suffices to show that the intersection of infinitely many height $1$ prime ideals of $R$ is trivial. So let $\cl S$ be an infinite set of prime ideals of $R$ of height $1$ and define
\[ I:=\bigcap_{\fr p\in\cl S}\fr p. \]
If $I\not=\{0\}$ then every $\fr p\in\cl S$, having height $1$, is minimal among the prime ideals of $R$ containing $I$. But the set of such prime ideals of $R$ has only finitely many minimal elements by \cite[Exercise 4.12]{mat}, and this is a contradiction. \end{proof}

Another by-product of Theorem \ref{thm1} is the following

\begin{prop} \label{prop2}
Let $R$ be a local noetherian domain with maximal ideal $\fr m$ and let $\rho_V$, $\rho_L$ be as above. Suppose that the representation $\bar\rho_{L,\fr m}$ is irreducible and that there exists a set $\cl S$ of prime ideals of $R$ such that
\begin{itemize}
\item[(i)] $\bigcap_{\fr p\in\cl S}\fr p=\{0\}$; 
\item[(ii)] $R/\fr p$ is a discrete valuation ring for every $\fr p\in\cl S$.
\end{itemize}
Then $\rho_V$ is irreducible.
\end{prop}

\begin{proof} The representation $\bar\rho_{L,\fr m}$ coincides with $\tilde\rho_{L,\fr m}$ and so it naturally identifies, for every prime ideal $\fr p\in\cl S$, with the reduction of $\tilde\rho_{L,\fr p}$ modulo the maximal ideal $\fr m/\fr p$ of the discrete valuation ring $R/\fr p$. Since we are assuming that $\bar\rho_{L,\fr m}$ is irreducible, Proposition \ref{dvr-prop} ensures that $\tilde\rho_{L,\fr p}$ is irreducible for every $\fr p\in\cl S$. In other words, condition (ii') is satisfied, and the irreducibility of $\rho_V$ follows from Theorem \ref{thm1}. \end{proof}  

\subsection{Representations over regular local rings} \label{regular-subsec}

In this subsection we consider a regular local ring $R$ with maximal ideal $\fr m$ and residue field $\kappa:=R/\fr m$. It is known that such a ring is necessarily a domain (see, e.g., \cite[Theorem 14.3]{mat}). If one also assumes that $R$ is complete with respect to its $\fr m$-adic topology (which we will not do) then classical structure theorems of Cohen (\cite{cohen}) give very precise descriptions of $R$ in terms of formal power series rings over either $\kappa$ or a complete discrete valuation ring of characteristic zero with residue field $\kappa$. Recall that a regular ring is a noetherian ring such that every localization $R_\fr p$ of $R$ at a prime ideal $\fr p$ is a regular local ring. If $R$ is a regular (respectively, regular local) ring then any formal power series ring over $R$ is a regular (respectively, regular local) ring as well (\cite[Theorem 19.5]{mat}). 

\begin{teo} \label{thm2}
With notation as above, let $R$ be a regular local ring. If $\bar\rho_{L,\fr m}$ is irreducible then $\rho_V$ is irreducible.
\end{teo}

\begin{proof} We proceed by induction on the Krull dimension $d$ of $R$. If $d=1$ then $R$ is a discrete valuation ring, hence the claim of the theorem follows from Proposition \ref{dvr-prop}. Now suppose that the theorem has been proved for rings of dimension $t$ and take $d=t+1$. Write $\fr m=(x_1,\dots,x_{t+1})$ and consider the infinite set $\cl S$ of prime ideals of $R$ given by
\[ \cl S:=\bigl\{(x_2-x_1^i)\mid i\geq1\bigr\}. \]
We want to prove that $\bar\rho_{L,\fr p}$ is irreducible for every $\fr p\in\cl S$. If $\fr p\in\cl S$ then the noetherian local ring $R/\fr p$ is regular of dimension $t$ (\cite[Theorem 14.2]{mat}); moreover, the reduction of $\tilde\rho_{L,\fr p}$ modulo the maximal ideal $\fr m/\fr p$ coincides with $\bar\rho_{L,\fr m}$, which is irreducible by hypothesis. The inductive assumption then gives the irreducibility of $\tilde\rho_{L,\fr p}$, which is equivalent to the irreducibility of $\bar\rho_{L,\fr p}$. Finally, since all prime ideals in $\cl S$ have height $1$ (for example, by Krull's principal ideal theorem), the irreducibility of $\rho_V$ follows from Corollary \ref{coro1}. \end{proof} 

Notice that this theorem extends the result proved in Proposition \ref{dvr-prop} for discrete valuation rings to the much larger class of regular local rings.

\begin{coro} \label{regular-coro}
{\rm (1)} Let $R$ be a noetherian domain and suppose that there exists a prime ideal $\fr p$ of $R$ such that 
\begin{itemize}
\item[(i)] $R_\fr p$ is a regular local ring;
\item[(ii)] $\bar\rho_{L,\fr p}$ is irreducible.
\end{itemize}
Then $\rho_V$ is irreducible.

{\rm (2)} If $R$ is a regular domain and there exists a prime ideal $\fr p$ of $R$ such that $\bar\rho_{L,\fr p}$ is irreducible then $\rho_V$ is irreducible.
\end{coro}

\begin{proof} Since, by definition of a regular ring, part (2) is an immediate consequence of part (1), it suffices to prove the first claim. Since $R_\fr p$ is a regular local ring and $\bar\rho_{L,\fr p}$ is irreducible, the representation $\rho_{L,\fr p}$ over $R_\fr p$ is irreducible by Theorem \ref{thm2}. But the fraction field of $R_\fr p$ is equal to $K$, hence the corollary follows from Proposition \ref{prop1}. \end{proof}

In particular, this corollary applies to the special case where $R$ is a Dedekind domain (e.g., the ring of algebraic integers of a number field).

\section{Arithmetic applications} \label{arithmetic-section}

We offer applications of Theorem \ref{thm1} (or, rather, of Theorem \ref{thm2}) in situations of arithmetic interest.

\subsection{Universal deformations of residual representations} \label{universal-subsec}

We briefly recall the basic definitions about deformations of residual representations; for more details, the reader is referred to the original article \cite{mazur} by Mazur and to the survey papers \cite{gouvea} and \cite{mazur2}.

Fix a prime number $p$ and a finite field $\kappa$ of characteristic $p$. Let $\Pi$ be a profinite group satisfying Mazur's finiteness condition $\Phi_p$ (\cite[\S 1.1]{mazur}), i.e. such that for every open subgroup $\Pi_0$ of $\Pi$ there are only finitely many continuous homomorphisms from $\Pi_0$ to the field $\F_p$ with $p$ elements. A remarkable example of a group with this property is represented by the Galois group $G_{\Q,S}$ over $\Q$ of the maximal field extension of $\Q$ which is unramified outside a finite set $S$ of primes of $\Q$. 

If $n\geq1$ is an integer, by a \emph{residual representation of dimension $n$} (of $\Pi$ over $\kappa$) we shall mean a continuous representation
\[ \bar\rho:\Pi\longrightarrow\GL_n(\kappa). \]
Write $\cl C$ for the category of \emph{coefficient rings} in the sense of Mazur (\cite[\S 2]{mazur2}), whose objects are complete noetherian local rings with residue field $\kappa$ and whose morphisms are (local) homomorphisms of complete local rings inducing the identity on $\kappa$. More generally, for an object $\Lambda\in\cl C$ we can consider the category $\cl C_\Lambda$ of \emph{coefficient $\Lambda$-algebras}, whose objects are complete noetherian local $\Lambda$-algebras with residue field $\kappa$ and whose morphisms are coefficient-ring homomorphisms which are also $\Lambda$-algebra homomorphisms. Observe that $\cl C=\cl C_{W(\kappa)}$ where $W(\kappa)$ is the ring of Witt vectors of $\kappa$, i.e. the (unique) unramified extension of $\Z_p$ with residue field $\kappa$.

If $A\in\cl C_\Lambda$ then two continuous representations
\[ \rho_1,\rho_2:\Pi\longrightarrow\GL_n(A) \]
will be said to be \emph{strictly equivalent} if there exists $M$ in the kernel of the reduction map $\GL_n(A)\rightarrow\GL_n(\kappa)$ such that $\rho_1=M\rho_2M^{-1}$. Given a residual representation $\bar\rho$ as above, a \emph{deformation} of $\bar\rho$ to $A\in\cl C$ is a strict equivalence class $\boldsymbol\rho$ of (continuous) representations
\[ \rho:\Pi\longrightarrow\GL_n(A) \]
which reduce to $\bar\rho$ via the map $\GL_n(A)\rightarrow\GL_n(\kappa)$. By abuse of notation, we will write
\[ \boldsymbol\rho:\Pi\longrightarrow\GL_n(A) \]
to denote a deformation of $\bar\rho$ to $A$. 

Finally, given an $n$-dimensional residual representation $\bar\rho$, endow $\kappa^n$ with a (left) $\Pi$-module structure via $\bar\rho$ and define
\[ C(\bar\rho):=\Hom_\Pi(\kappa^n,\kappa^n) \]
to be its ring of $\Pi$-module endomorphisms.

\begin{teo}[Mazur, Ramakrishna] \label{mr-thm}
With notation as above, let $\Lambda\in\cl C$ and let
\[ \bar\rho:\Pi\longrightarrow\GL_n(\kappa) \]
be a residual representation such that $C(\bar\rho)=\kappa$. Then there exists a ring $\cl R_\Lambda=\cl R_\Lambda(\Pi,\kappa,\bar\rho)\in\cl C_\Lambda$ and a deformation
\[ \boldsymbol\rho:\Pi\longrightarrow\GL_n(\cl R_\Lambda) \]
of $\bar\rho$ to $\cl R$ such that any deformation of $\bar\rho$ to a ring $A\in\cl C_\Lambda$ is obtained from $\boldsymbol\rho$ via a unique morphism $\cl R_\Lambda\rightarrow A$.
\end{teo}

For a proof, see \cite[Theorem 3.10]{gouvea}. We call $\cl R_\Lambda$ the \emph{universal deformation ring} and $\boldsymbol\rho$ the \emph{universal deformation} of $\bar\rho$. 

\begin{rem} \label{absolutely-irr-rem}
By Schur's lemma, $C(\bar\rho)=\kappa$ if $\bar\rho$ is absolutely irreducible, so any absolutely irreducible residual representation admits universal deformation.
\end{rem}

Let now $\ad(\bar\rho)$ be the adjoint representation of $\bar\rho$, i.e. the $\kappa$-vector space $\M_n(\kappa)$ on which $\Pi$ acts on the left by conjugation via $\bar\rho$ (hence $C(\bar\rho)$ is non-canonically isomorphic to the subspace of $\Pi$-invariants of $\ad(\bar\rho)$). The following result is part of a theorem of Mazur, for a proof of which we refer to \cite[Proposition 2]{mazur} or \cite[Theorem 4.2]{gouvea}.

\begin{teo}[Mazur] \label{mazur-thm}
Suppose $C(\bar\rho)=\kappa$, let $\cl R_\Lambda$ be the universal deformation ring of $\bar\rho$ and define 
\begin{equation} \label{d-eq}
d_1:=\dim_\kappa H^1\bigl(\Pi,\ad(\bar\rho)\bigr),\qquad d_2:=\dim_\kappa H^2\bigl(\Pi,\ad(\bar\rho)\bigr).
\end{equation}
If $d_2=0$ then $\cl R_\Lambda\simeq\Lambda[\![x_1,\dots,x_{d_1}]\!]$. 
\end{teo}

When $d_2=0$ we say that the deformation problem for $\bar\rho$ is \emph{unobstructed}. It is convenient to introduce the following terminology.

\begin{defi} \label{irreducible-def-defi}
Suppose that the coefficient ring $A$ is a domain. A deformation $\boldsymbol\rho$ of $\bar\rho$ to $A$ is said to be \emph{irreducible} if every $\rho\in\boldsymbol\rho$ is irreducible.
\end{defi}

Here the irreducibility of a representation over $A$ is understood in the sense of \S \ref{terminology-subsection}. Of course, since all representations in a deformation of $\bar\rho$ to $A$ are equivalent, Definition \ref{irreducible-def-defi} amounts to requiring that there exists $\rho\in\boldsymbol\rho$ which is irreducible.

Let now $\cl O$ be a complete regular local ring with residue field $\kappa$. As an easy consequence of the results in \S \ref{regular-subsec}, we obtain 
 
\begin{prop} \label{irreducible-def-prop}
With notation as before, let
\[ \bar\rho:\Pi\longrightarrow\GL_n(\kappa) \]
be an $n$-dimensional irreducible residual representation such that $C(\bar\rho)=\kappa$ and $d_2=0$. The universal deformation
\[ \boldsymbol\rho:\Pi\longrightarrow\GL_n(\cl R_{\cl O}) \]
of $\bar\rho$ is irreducible.
\end{prop} 

\begin{proof} Since we are in the unobstructed case, Theorem \ref{mazur-thm} ensures that there is an $\cl O$-algebra isomorphism $\cl R_{\cl O}\simeq\cl O[\![x_1,\dots,x_{d_1}]\!]$ with $d_1$ as in \eqref{d-eq}. In particular, the deformation ring $\cl R_{\cl O}$ is a regular local ring, whose maximal ideal we denote by $\fr m$. Moreover, every representation $\rho\in\boldsymbol\rho$ reduces modulo $\fr m$ to $\bar\rho$, which is irreducible by assumption, hence the proposition follows from Theorem \ref{thm2}. \end{proof}   

\begin{rem}
The reader may have noticed that Proposition \ref{irreducible-def-prop} has not been stated in the greatest generality. In fact, all we really need to know in order to deduce the irreducibility of $\boldsymbol\rho$ from Theorem \ref{thm2} is that $\bar\rho$ is irreducible and admits a universal deformation ring which is a regular local ring. However, the conditions appearing in Proposition \ref{irreducible-def-prop} are the ones that can be checked more easily in ``practical'' situations, so we preferred to formulate our results in a slightly less general but more readily applicable form.
\end{rem} 

\subsection{Deformations of residual modular representations} \label{modular-subsec}

Of remarkable arithmetic interest are the residual Galois representations associated with modular forms, and now we want to specialize Proposition \ref{irreducible-def-prop} to this setting. 

Let $f$ be a newform of level $N$ and weight $k\geq2$ and let $K_f$ be the number field generated by the Fourier coefficients of $f$. For every prime $\lambda$ of $K_f$ Deligne has associated with $f$ a semisimple representation
\[ \bar\rho_{f,\lambda}:G_{\Q,S}\longrightarrow\GL_2(\kappa_\lambda) \]
over the residue field $\kappa_\lambda$ of $K_f$ at $\lambda$; here $G_{\Q,S}$ is the Galois group over $\Q$ of the maximal extension of $\Q$ unramified outside the finite set $S$ of places dividing $N\ell\infty$ where $\ell$ is the characteristic of $\kappa_\lambda$ and $\infty$ denotes the unique archimedean place of $\Q$. The representation $\bar\rho_{f,\lambda}$ is absolutely irreducible for all but finitely many primes $\lambda$; for such a $\lambda$ let $\cl R_{f,\lambda}^S$ be the universal deformation ring parametrizing deformations of $\bar\rho_{f,\lambda}$ to complete noetherian local rings with residue field $\kappa_\lambda$ (cf. Remark \ref{absolutely-irr-rem}). 

As a sample result in the context of residual modular representations, we prove 

\begin{prop} \label{modular-prop}
If $k>2$ then the universal deformation
\[ \boldsymbol\rho_{f,\lambda}:G_{\Q,S}\longrightarrow\GL_2\bigl(\cl R_{f,\lambda}^S\bigr) \]
of $\bar\rho_{f,\lambda}$ is irreducible for all but finitely many primes $\lambda$ of $K_f$, while if $k=2$ then $\boldsymbol\rho_{f,\lambda}$ is irreducible for a subset of primes $\lambda$ of $K_f$ of density $1$.
\end{prop}

\begin{proof} By a theorem of Weston (\cite[Theorem 1]{weston}), if $k>2$ (respectively, $k=2$) then the deformation problem for $\bar\rho_{f,\lambda}$ is unobstructed for almost all primes $\lambda$ of $K_f$ (respectively, for a subset of primes $\lambda$ of $K_f$ of density $1$). Moreover, for every such $\lambda$ there is an isomorphism
\[ \cl R_{f,\lambda}^S\simeq W(\kappa_\lambda)[\![x_1,x_2,x_3]\!] \]
where $W(\kappa_\lambda)$ is the ring of Witt vectors of $\kappa_\lambda$. Since $\bar\rho_{f,\lambda}$ is (absolutely) irreducible, the proposition follows from Theorem \ref{thm2}. \end{proof}

See \cite{weston2} for explicit results on the set of obstructed primes for $f$ in the case where the level $N$ is square-free.

\end{document}